\documentclass[a4paper,12pt]{article}
\usepackage[pdftex]{color,graphicx}
\usepackage{amssymb}
\usepackage{amsfonts}
\usepackage{amsmath}
\usepackage{epsfig}
\usepackage{amsthm}
\usepackage{epsfig}
\usepackage{pstricks}
\usepackage{draftcopy}
\usepackage[latin1]{inputenc}

\newtheorem{proposition}{Proposition}[section]
\newtheorem{lemma}{Lemma}[section]

\newtheorem{example}{Example}[section]
\newtheorem*{question}{Question}
\newtheorem{corollary}{Corollary}[section]

\newtheorem{remark}{Remark}[section]
\newtheorem*{acknowledgements}{Acknowledgements}
\newtheorem*{teoA}{Theorem A}
\newtheorem*{teoB}{Theorem B}
\newtheorem*{teoC}{Theorem C}
\newcommand{\R}{\mathbb{R}}

\newcommand{\s}{\mathbb{S}}

\newcommand{\B}{\mathbb{B}}

\newcommand{\be}{\beta}
\newcommand{\al}{\alpha}
\newcommand{\ria}{\rightarrow}

\newcommand{\om}{\omega}
\newcommand{\n}{\nabla}

\newcommand{\ran}{\rangle}
\newcommand{\lan}{\langle}
\newcommand{\ve}{\varepsilon}
\newcommand{\vp}{\varphi}

\DeclareMathOperator{\vol}{vol}
\DeclareMathOperator{\di}{div}
\DeclareMathOperator{\tr}{tr}

\DeclareMathOperator{\sen}{sen}

\numberwithin{equation}{section}
\title{On Stable Hypersurfaces with Vanishing Scalar Curvature}
\author{Greg\' orio Silva Neto}
\date{May, 24 2013}

\begin{document}

\maketitle

\begin{abstract}
We will prove that \emph{there are no stable complete hypersurfaces of $\R^4$ with zero scalar curvature, polynomial volume growth and such that $\dfrac{(-K)}{H^3}\geq c>0$ everywhere, for some constant $c>0$}, where $K$ denotes the Gauss-Kronecker curvature and $H$ denotes the mean curvature of the immersion. Our second result is the Bernstein type one \emph{there is no entire graphs of $\R^4$ with zero scalar curvature such that $\dfrac{(-K)}{H^3}\geq c>0$ everywhere}. 
At last, it will be proved that, if there exists a stable hypersurface with zero scalar curvature and $\dfrac{(-K)}{H^3}\geq c>0$ everywhere, that is, with volume growth greater than polynomial, then its tubular neighborhood is not embedded for suitable radius.
\end{abstract}

\section{Introduction}
\label{intro}

Let $x:M^3\ria\R^4$ be an isometric immersion. If $\lambda_1,\lambda_2,\lambda_3$ are the eigenvalues of the second fundamental form, then the \emph{scalar curvature} $R$, the non-normalized \emph{mean curvature} $H,$ and the \emph{Gauss-Kronecker curvature} $K$ are given, respectively, by 

\begin{equation}\label{eq.curv}
R=\lambda_1\lambda_2+\lambda_1\lambda_3+\lambda_2\lambda_3, \ \ H=\lambda_1+\lambda_2+\lambda_3 \ \ \mbox{and} \ K=\lambda_1\lambda_2\lambda_3.
\end{equation}

In 1959, Hartman and Nirenberg, cf. \cite{hart}, have shown that the only surfaces with zero Gaussian curvature in three-dimensional Euclidean space are planes and cylinders.

Generalizing this fact, in 1977, Cheng and Yau, cf. \cite{yau}, showed that the only complete non-compact hypersurfaces with constant scalar curvature and non-negative sectional curvature in the Euclidean space $\R^{n+1}$ are the generalized cylinders $\s^{n-p}\times\R^p.$

Let $D\subset M^3$ be a regular domain, i.e., a domain with compact closure and piecewise smooth boundary. A \emph{compact supported variation} of the immersion $x$ is a differentiable map $X:(-\varepsilon,\varepsilon)\times D\ria\R^4, \ \ve>0,$ such that, for each $t\in (-\ve,\ve),$ $X_t:D\ria\R^4,$ $X_t(p)=X(t,p)$ is an immersion, $X_0=x|_D$ and $X_t|_{\partial D}=X_0|_{\partial D}.$ We recall that hypersurfaces of $\R^4$ with zero scalar curvature are critical points of the functional
$$
\mathcal{A}_1(t) = \int_M H(t)dM_t
$$
under all variations compactly supported in $D$ (see \cite{reilly}, \cite{AdCC}, \cite{rosen}, \cite{BC}).

Following Alencar, do Carmo, and Elbert, cf. \cite{AdCE}, let us define the concept of stability for immersions with zero scalar curvature. Let $A:TM\ria TM$ the linear operator associated to the second fundamental form  of immersion $x.$ We define the \emph{first Newton transformation} $P_1:TM\ria TM$ by $P_1 = HI - A,$ where $I$ denotes identity operator. We now introduce a second order differential operator which will play a role similar to that of Laplacian in the minimal case:
\begin{equation}\label{def.L1}
L_1(f) = \di(P_1(\n f)),
\end{equation}
where $\di X$ denotes the divergence of vector field X, and $\n f$ denotes the gradient of the function $f$ in the induced metric. In \cite{HL}, Hounie and Leite showed that $L_1$ is elliptic if and only if rank $A>1.$ Thus, $K\neq0$ everywhere implies $L_1$ is elliptic, and if $H>0,$ then $P_1$ is a positive definite linear operator.

Computing the second derivative of functional $\mathcal{A}_1$ we obtain
$$
\left.\dfrac{d^2\mathcal{A}_1}{dt^2}\right|_{t=0} = -2\int_M f(L_1 f - 3Kf)dM,
$$
where $f=\left\lan\frac{dX}{dt}(0),\eta\right\ran,$ and $\eta$ is the normal vector field of the immersion.

Since $H^2=|A|^2 + 2R,$ if $R=0$ then $H^2=|A|^2,$ i.e, if $K\neq0$ everywhere, then $H^2=|A|^2 \neq0$ everywhere. It implies that $H>0$ everywhere or $H<0$ everywhere. Hence, unlike minimal case, the sign of functional $\mathcal{A}_1$ depends on choice of orientation of $M^3.$ Following Alencar, do Carmo and Elbert, see \cite{AdCE}, if we choose an orientation such that $H>0$ everywhere, then the immersion will be \emph{stable} if $\left.\dfrac{d^2\mathcal{A}_1}{dt^2}\right|_{t=0}>0$ under all compact support variations. Otherwise, i. e., if we choose an orientation such that $H<0,$ then $x$ is stable if $\left.\dfrac{d^2\mathcal{A}_1}{dt^2}\right|_{t=0}<0.$ For more details, see \cite{AdCE}.

In the pursuit of  this subject, Alencar, do Carmo and Elbert, cf. \cite{AdCE}, have posed the following:
\begin{question}\label{ques}
Is there any stable complete hypersurface $M^3$ in $\R^4$ with zero scalar curvature and everywhere non-zero Gauss-Kronecker curvature?
\end{question}
The goal of this paper is to give some partial answers to this question.
Let $B_r(p)$ be the geodesic ball with center $p\in M$ and radius $r.$ We say that a Riemannian manifold $M^3$ has \emph{polynomial volume growth}, if there exists $\alpha\in[0,4]$ such that
\begin{equation}\label{def.poli.growth}
\dfrac{\vol(B_r(p))}{r^\al}<\infty,
\end{equation}
for all $p\in M.$

A well known inequality establishes that
\begin{equation}
HK\leq \dfrac{1}{2}R^2.
\end{equation}
If $R=0$ and $K\neq0$ everywhere, then the quotient $\dfrac{K}{H^3}$ is always negative, independent on choice of orientation. Furthermore, considering $K$ and $H^3$ as functions of the eigenvalues of second fundamental form, we can see that
$$
0< \dfrac{(-K)}{H^3}\leq \dfrac{4}{27},
$$
provided $K$ and $H^3$ are homogeneous polynomials of degree $3.$ For details, see Appendix.

The first result is

\begin{teoA}
There is no stable complete hypersurface $M^3$ of $\R^4$ with zero scalar curvature, polynomial volume growth and such that $$\dfrac{(-K)}{H^3}\geq c>0$$ everywhere, for some constant $c>0.$ Here $H$ denotes the mean curvature and $K$ denotes the Gauss-Kronecker curvature of the immersion.
\end{teoA}

As a consequence of Theorem A, we obtain the following Bernstein type result. 

\begin{teoB}\label{teo_graph1}
There are no entire graphs $M^3$ of $\R^4$ with zero scalar curvature and such that $$\dfrac{(-K)}{H^3}\geq c>0$$ everywhere, for some constant $c>0.$ Here $H$ denotes the mean curvature and $K$ denotes the Gauss-Kronecker curvature of the immersion.
\end{teoB}

Following Nelli and Soret, cf. \cite{nelli}, in section $5$ we show that, if $M^3$ is a stable complete hypersurface of $\R^4$ with zero scalar curvature and such that $\dfrac{(-K)}{H^3}\geq c>0$ everywhere, then the tube around $M$ is not embedded for suitable radius. Precisely, we define the \emph{tube of radius} $h$ \emph{around} $M$ the set
$$
T(M,h)=\{x\in\R^4;\ \exists\ p\in M,\ x=p+ t\eta,\ t\leq h(p)\}
$$
where $\eta$ is the normal vector of second fundamental form of the immersion and $h:M\ria\R$ is an everywhere non-zero smooth function. We prove

\begin{teoC}
Let $M^3$ be a stable complete hypersurface of $\R^4$ with vanishing scalar curvature. Suppose that the second fundamental form of the immersion is bounded and there exists a constant $c>0$ such that $\dfrac{(-K)}{H^3}\geq c>0$ everywhere. Then, for constants $0<b_1\leq1, \ b_2>0,$ and for any smooth function $h:M\ria\R$ satisfying
$$
h(p)\geq\min\left\{\dfrac{b_1}{|A(p)|},b_2\rho(p)^\delta\right\}, \ p\in M,\ \delta>0,
$$
the tube $T(M,h)$ is not embedded. Here, $\rho(p)$ denotes the intrinsic distance in $M$ to a fixed point $p_0\in M.$
\end{teoC}

\begin{acknowledgements} {\normalfont I would like to thank professor Hil\' ario Alencar for read critically this manuscript and for his many valuable suggestions. I would also to thank Detang Zhou by his useful suggestions, and professor Barbara Nelli for her suggestions to clarify some arguments used in the proof of Theorem C.} 
\end{acknowledgements}

\section{Preliminary Results}
\label{sec:1}

Let $B(X,Y)=\overline{\n}_XY - \n_XY$ be the second fundamental form of immersion $x,$ where $\n$ and $\overline{\n}$ are the connections of $M^3$ and $\R^4,$ respectively. The \emph{shape operator} is the only symmetric linear operator $A:TM\ria TM$ such that 
$$
B(X,Y)=\lan A(X),Y\ran\eta, \ \forall\ X,Y \in TM,
$$
where $\eta$ is the normal field of the immersion $x.$
 
Denote by $|A|^2=\tr (A^2)$ the matrix norm of second fundamental form. Since $H^2=|A|^2+2R,$ if $R=0,$ then $H^2=|A|^2.$ Hence, $K\neq0$ everywhere implies $H=|A|\neq0$ everywhere, and we can choose an orientation of $M$ such that $H>0$ everywhere.

\begin{remark}
From now on, let us fix an orientation of $M^3$ such that $H>0$ everywhere.
\end{remark}

A well known inequality establishes that
$$
HK\leq \dfrac{1}{2}R^2.
$$

Therefore, by using inequality above, $R=0$ and $H>0$ everywhere implies $K<0$ everywhere.

Define $P_1:TM\ria TM$ by $P_1=HI-A$ the \emph{first Newton transformation}. If $R=0$ and $H>0,$ then $P_1$ is positive definite. It was proved by Hounie and Leite in a general point of view, see \cite{HL}. In fact, $P_1$ positive definite implies $L_1(f) = \di(P_1(\n f))$ is an elliptic differential operator. Let us give here a proof for sake of completeness. It suffices to prove that $H-\lambda_i>0, \ i=1,2,3.$ In fact,
$$
\lambda_1^2(H-\lambda_1) = \lambda_1^2(\lambda_2+\lambda_3) = \lambda_1^2\lambda_2 + \lambda_1^2\lambda_3.
$$
Since $R=\lambda_1\lambda_2+\lambda_1\lambda_3+\lambda_2\lambda_3=0,$ we have
$$
0=\lambda_1R=\lambda_1(\lambda_1\lambda_2+\lambda_1\lambda_3+\lambda_2\lambda_3)= \lambda_1^2\lambda_2+\lambda_1^2\lambda_3+\lambda_1\lambda_2\lambda_3,
$$
i.e.,
$$
\lambda_1^2\lambda_2+\lambda_1^2\lambda_3 = - \lambda_1\lambda_2\lambda_3 = -K >0.
$$
Thus,
$$
\lambda_1^2(H-\lambda_1)=\lambda_1^2\lambda_2 + \lambda_1^2\lambda_3=- \lambda_1\lambda_2\lambda_3>0,
$$
and then, $H-\lambda_1>0.$ The other cases are analogous.

Our choice of orientation, i.e, that one such that $H>0$ everywhere, implies stability condition is equivalent to
\begin{equation}\label{ineq.stab}
-3\int_M K f^2 dM \leq \int_M\lan P_1(\n f),\n f\ran dM.
\end{equation}
The inequality $(\ref{ineq.stab})$ is known as \emph{stability inequality}.

\begin{remark} When $H<0,$ then $K>0$ and $P_1$ is negative definite. In this case, stability condition is equivalent to 
$$
3\int_M K f^2 dM \leq \int_M \lan (-P_1)(\n f),\n f\ran dM.
$$
\end{remark}

Let $\n A (X,Y,Z):=\lan \n_Z(A(X))-A(\n_ZX),Y \ran$ be the \emph{covariant derivative} of operator $A.$ The following proposition will play an important role in the proof of main theorems. In \cite{dCP}, do Carmo and Peng showed a very similar inequality for minimal hypersurfaces. 

\begin{proposition}\label{lema2}
If $R=0$ and there exists $c>0$ such that $\frac{-K}{H^3}\geq c>0$ everywhere, then there exists $c_0>0,$ depending on $c,$ such that
$$
|\n A|^2 - |\n H|^2\geq \dfrac{2}{1+2c_0^2}|\n H|^2,
$$
where $\n H$ denotes the gradient of $H.$
\end{proposition}

\begin{proof}
Let us fix $p\in M$ and choose $\{e_1(p),e_2(p),e_3(p)\}$ an orthonormal basis of $T_pM$ such that $h_{ij}(p)=\lambda_i(p)\delta_{ij},$ where $h_{ij}=\lan A(e_i),e_j\ran,$ $\lambda_i(p)$ denotes the eigenvalues of $A$ in $p$ and $\delta_{ij}$ is the \emph{Kronecker delta}
$$
\delta_{ij}=\left\{\begin{array}{rcl}
1&\mbox{if}&i=j;\\
0&\mbox{if}&i\neq j.
\end{array}\right.
$$

Extending this basis by parallel transport along geodesics starting on $p,$ to a referential in a neighbourhood of $p,$  we have $\n_{e_i(p)}e_j(p)=0,$ for all $i,j=1,2,3.$ This is called \emph{geodesic referential at $p$}. 

Let us denote by $h_{ij;k}=(h_{ij})_k:=e_k(h_{ij})$ the covariant derivatives of function $h_{ij},$ and by $h_{ijk}$ the components of tensor $\n A$ in the referential $\{e_1, e_2, e_3\},$ i.e., $h_{ijk}=\n A(e_i,e_j,e_k).$ Since $\{e_1,e_2,e_3\}$ is a geodesic referential, we have
$$
\begin{array}{rcl}
h_{ijk}&=&\n A(e_i,e_j,e_k)=\lan\n_{e_k}(A(e_i))-A(\n_{e_k}e_i),e_j\ran=\lan\n_{e_k}(A(e_i)),e_j\ran\\
&&\\
       &=& e_k(\lan A(e_i),e_j\ran) - \lan A(e_i),\n_{e_k}e_j\ran = e_k(\lan A(e_i),e_j\ran)=e_k(h_{ij})\\
       &&\\
       &=&h_{ij;k}.\\
\end{array}
$$
%
Since $R=0,$ then $H^2=|A|^2.$ Using this fact, we have 
\[
\begin{split}
4H^2|\n H|^2&=|\n(H^2)|^2 = |\n (|A|^2)|^2=\sum_{k=1}^3\left[\left(\sum_{i,j=1}^3 h_{ij}^2\right)_k\right]^2\\
						&=\sum_{k=1}^3\left(\sum_{i,j=1}^32h_{ij}h_{ij;k}\right)^2 				  =4\sum_{k=1}^3\left(\sum_{i=1}^3h_{ii}h_{ii;k}\right)^2.\\
\end{split}
\]
Now, using Cauchy-Schwarz inequality, we obtain
\[
\begin{split}
4\sum_{k=1}^3\left(\sum_{i=1}^3h_{ii}h_{ii;k}\right)^2& 
\leq 4\sum_{k=1}^3\left[\left(\sum_{i=1}^3h_{ii}^2\right)\left(\sum_{i=1}^3h_{ii;k}^2\right)\right]\\
&=4|A|^2\left(\sum_{i,k=1}^3h_{ii;k}^2\right) = 4H^2\left(\sum_{i,k=1}^3h_{ii;k}^2\right).\\
\end{split}
\]
Therefore,
\begin{equation}\label{ineq.lema2.1}
|\n H|^2\leq\sum_{i,k=1}^3h_{ii;k}^2.
\end{equation}
On the other hand, since $R=h_{11}h_{22}+h_{11}h_{33}+h_{22}h_{33} - h_{12}^2 - h_{13}^2 - h_{23}^2=0,$ we have for $k=1,2,3,$
\[
\begin{split}
0&=(h_{11}h_{22}+h_{11}h_{33}+h_{22}h_{33}- h_{12}^2 - h_{13}^2 - h_{23}^2)_k\\
&=h_{11k}h_{22}+h_{11}h_{22k}+h_{11k}h_{33}+h_{11}h_{33k}+h_{22k}h_{33}+h_{22}h_{33k}\\
&\qquad -2h_{12}h_{12k} - 2h_{13}h_{13k} - 2h_{23}h_{23k}\\
&=h_{11k}h_{22}+h_{11}h_{22k}+h_{11k}h_{33}+h_{11}h_{33k}+h_{22k}h_{33}+h_{22}h_{33k}\\
&=h_{11k}(h_{22}+h_{33}) + h_{22k}(h_{11}+h_{33}) + h_{33k}(h_{11}+h_{22})\\
&=h_{11k}(H-h_{11}) + h_{22k}(H-h_{22}) + h_{33k}(H-h_{33})\\
\end{split}
\]
Thus, taking $k=1$ in the inequality above, we have
\[
h_{111}=-\dfrac{1}{H-h_{11}}[h_{221}(H-h_{22}) + h_{331}(H-h_{33})]
\]
Analogously, taking $k=2$ and $k=3$ we have
\[
h_{222}=-\dfrac{1}{H-h_{22}}[h_{112}(H-h_{11}) + h_{332}(H-h_{33})]
\]
and
\[
h_{333}=-\dfrac{1}{H-h_{33}}[h_{113}(H-h_{11}) + h_{223}(H-h_{22})].
\]
Squaring and summing, we have
\[
\begin{split}
h_{111}^2+h_{222}^2+h_{333}^2&=\dfrac{1}{(H-h_{11})^2}[h_{221}(H-h_{22}) + h_{331}(H-h_{33})]^2\\
&\qquad\dfrac{1}{(H-h_{22})^2}[h_{112}(H-h_{22}) + h_{332}(H-h_{33})]^2\\
&\qquad\dfrac{1}{(H-h_{33})^2}[h_{113}(H-h_{11}) + h_{223}(H-h_{22})]^2\\
\end{split}
\]
Using inequality $(a+b)^2=a^2+2ab+b^2\leq 2(a^2+b^2),$ we have
\[
\begin{split}
h_{111}^2+h_{222}^2+h_{333}^2&\leq 2\left[\left(\dfrac{H-h_{22}}{H-h_{11}}\right)^2 h_{221}^2 + \left(\dfrac{H-h_{33}}{H-h_{11}}\right)^2 h_{331}^2 \right.\\
&\left.\qquad \left(\dfrac{H-h_{11}}{H-h_{22}}\right)^2 h_{112}^2+\left(\dfrac{H-h_{33}}{H-h_{22}}\right)^2 h_{332}^2\right.\\
&\left.\qquad\left(\dfrac{H-h_{11}}{H-h_{33}}\right)^2 h_{113}^2 + \left(\dfrac{H-h_{22}}{H-h_{33}}\right)^2 h_{223}^2 \right]\\
\end{split}
\]
Since the functions \[g_{ij}:\R^3\ria\R, \ g_{ij}(h_{11},h_{22},h_{33})=\left(\dfrac{H-h_{ii}}{H-h_{jj}}\right)^2, \ i,j=1,2,3\] are quotients of homogeneous polynomials of same degree, the values of $g_{ij}$ depends only on its value in the unit sphere $\s^2.$ Since $\{(h_{11},h_{22},h_{33})\in\R^3| R=0\}$ is closed in $\R^3,$ $\{(h_{11},h_{22},h_{33})\in\R^3|\frac{-K}{H^3}\geq c>0\}=\{(h_{11},h_{22},h_{33})\in\s^2|\frac{-K}{H^3}\geq c>0\}$ and $\s^2$ are compact sets of $\R^3,$ their intersection is a compact set of $\s^2.$ Thus all the functions $g_{ij}$ has a maximum and a minimum in $\s^2$. Let $c_0>0$ the maximum of the maxima of the functions $g_{ij}, i,j=1,2,3.$
Then
\[
h_{111}^2+h_{222}^2+h_{333}^2\leq2c_0^2\left(h_{112}^2 + h_{113}^2 + h_{221}^2 + h_{223}^2 + h_{331}^2 + h_{332}^2\right).
\]
This implies
\[
\begin{split}
|\n H|^2&\leq \sum_{i,k=1}^3 h_{iik}^2 = h_{111}^2 + h_{112}^2 + h_{113}^2+ h_{221}^2+ h_{222}^2+ h_{223}^2\\
        &\qquad + h_{331}^2+ h_{331}^2+ h_{333}^2\\
        &\leq (1+2c_0^2)\left( h_{112}^2 + h_{113}^2 + h_{221}^2+ h_{223}^2+ h_{331}^2+ h_{332}^2 \right)\\
        &\leq (1+2c_0^2)\left[\dfrac{1}{2}(h_{121}^2 + h_{211}^2) + \dfrac{1}{2}(h_{131}^2 + h_{311}^2) + \dfrac{1}{2}(h_{212}^2 + h_{122}^2) \right.\\
        &\left.\dfrac{1}{2}(h_{232}^2 + h_{322}^2) + \dfrac{1}{2}(h_{313}^2 + h_{133}^2) + \dfrac{1}{2}(h_{323}^2 + h_{223}^2)\right]\\
        &=\dfrac{1+2c_0^2}{2}\left(h_{121}^2 + h_{211}^2 + h_{131}^2 + h_{311}^2 + h_{212}^2 + h_{122}^2\right.\\
        &\left.\qquad\qquad\qquad+ h_{232}^2 + h_{322}^2 + h_{313}^2 + h_{133}^2 + h_{323}^2 + h_{233}^2\right).\\
\end{split}
\]
Therefore,
\[
\begin{split}
|\n A|^2=\sum_{i,j,k=1}^3 h_{ijk}^2&\geq\sum_{i,k=1}^3 h_{iik}^2+\sum_{i\neq k=1}^3 h_{iki}^2+\sum_{i\neq k=1}^3 h_{kii}^2\\
&\geq|\n H|^2 + \dfrac{2}{1+2c_0^2}|\n H|^2\\
&=\left(1+\dfrac{2}{1+2c_0^2}\right)|\n H|^2.\\
\end{split}
\]
\end{proof}

\section{Main Theorems}
\label{sec:2}
Hereafter, we will fix a point $p_0\in M$ and denote by $B_r$ the geodesic (intrinsic) ball of center $p_0$ and radius $r.$

The main tool to prove Theorem A stated in the Introduction is the following

\begin{proposition}\label{prop.sobolev2}
Let $x:M^3\ria\R^4$ be a stable isometric immersion with zero scalar curvature and such that $K$ is nowhere zero. Then, for all smooth function $\psi$ with compact support in $M,$ for all $\delta>0$ and $0<q<\sqrt{\frac{2}{1+2c_0^2}},$ there exists constants $\Lambda_1(q), \Lambda_2(q)>0$ such that
\begin{equation}\label{ineq1}
\int_M H^{5+2q}\left(\dfrac{(-K)}{H^3}- \Lambda_1\delta^{\frac{5+2q}{3+2q}}\right)\psi^{5+2q}dM \leq \Lambda_2\delta^{-\frac{5+2q}{2}}\int_M|\n\psi|^{5+2q}dM.
\end{equation}
\end{proposition}

\begin{proof}
Let us choose an orientation such that $H>0$ and apply the corresponding stability inequality
\begin{equation}\label{stab.ineq.2}
3\int_M (-K)f^2 dM \leq\int_M \lan P_1(\n f), \n f\ran dM,
\end{equation}
for $f=H^{1+q}\vp,$ where $q>0,$ and $\vp$ is a smooth function compactly supported on $M$. First note that
$$
\n f = \n(H^{1+q}\vp) = (1+q)H^q\vp\n H + H^{1+q}\n\vp.
$$
It implies
$$
\begin{array}{rcl}
\lan P_1(\n f),\n f\ran &=& (1+q)^2H^{2q}\vp^2\lan P_1(\n H),\n H\ran\\
&&\\
&& + 2(1+q)H^{1+2q}\vp\lan P_1(\n H),\n \vp\ran\\
&&\\
&&+ H^{2+2q}\lan P_1(\n \vp),\n \vp\ran.\\
\end{array}
$$

Since $H>0,$ then $P_1$ is positive definite. Now, let us estimate the second term in the right hand side of identity above. By using Cauchy-Schwarz inequality followed by inequality $xy\leq\dfrac{x^2}{2} + \dfrac{y^2}{2},$ for all $x,y\in\R,$ we obtain
\begin{equation} \label{ineq.sobolev.1}
\begin{split}
 H^{1+2q}\vp\lan P_1(\n H),\n\vp\ran &=H^{2q}\lan \sqrt{\beta}\vp\sqrt{P_1}(\n H),(1/\sqrt{\beta})H\sqrt{P_1}(\n\vp)\ran\\
&\leq H^{2q}\|\sqrt{\beta}\vp\sqrt{P_1}(\n H)\| \|(1/\sqrt{\beta})H\sqrt{P_1}(\n\vp)\|\\
&\leq H^{2q}\left(\dfrac{\|\sqrt{\beta}\vp\sqrt{P_1}(\n H)\|^2}{2}+\dfrac{\|(1/\sqrt{\beta})H\sqrt{P_1}(\n\vp)\|^2}{2}\right)\\
&=\dfrac{\beta}{2}H^{2q}\vp^2\lan P_1(\n H),\n H\ran + \dfrac{1}{2\beta}H^{2+2q}\lan P_1(\n\vp),\n\vp\ran,\\ 
\end{split}
\end{equation}
for any constant $\beta>0.$ Then stability inequality (\ref{stab.ineq.2}) becomes
\begin{equation}\label{eq.b}
\begin{array}{rcl}
\displaystyle{3\int_M(-K)H^{2+2q}\vp^2dM}&\leq&\displaystyle{(1+q)^2\int_M H^{2q}\vp^2\lan P_1(\n H),\n H\ran dM}\\
&&\\
&&\displaystyle{+ 2(1+q)\int_M H^{1+2q}\vp\lan P_1(\n H),\n \vp\ran dM}\\
&&\\
&&\displaystyle{+\int_M H^{2+2q}\lan P_1(\n \vp),\n \vp\ran dM}\\
&&\\
&\leq&\displaystyle{\left((1+q)^2+(1+q)\beta\right)\int_M H^{2q}\vp^2\lan P_1(\n H),\n H\ran dM}\\
&&\\
&&\displaystyle{+\left(1+\dfrac{(1+q)}{\beta}\right)\int_MH^{2+2q}\lan P_1(\n \vp),\n\vp\ran dM.}\\
\end{array}
\end{equation}

Let us estimate $\displaystyle{\int_MH^{2q}\vp^2 \lan P_1(\n H), \n H\ran dM}.$ By using identity
$$
L_1(fg) = \di(P_1(\n(fg))) = \di (f P_1(\n g)) + g L_1 f + \lan P_1(\n f),\n g\ran,
$$
we have
$$
\begin{array}{rcl}
L_1(H^{2+2q}\vp^2) &=& \di (H P_1(\n (H^{1+2q}\vp^2))) + H^{1+2q}\vp^2L_1(H)\\
&&\\
&& + \lan P_1(\n H), \n(H^{1+2q}\vp^2)\ran\\
&&\\
&=&\di(H P_1(\n(H^{1+2q}\vp^2))) + H^{1+2q}\vp^2L_1(H)\\
&&\\
&&+(1+2q)H^{2q}\vp^2\lan P_1(\n H),\n H\ran + 2H^{1+2q}\vp\lan P_1(\n H),\n \vp\ran.\\ 
\end{array}
$$
Integrating both sides of the identity above and by using Divergence Theorem, we obtain

$$
\begin{array}{rcl}
\displaystyle{(1+2q)\int_M H^{2q}\vp^2\lan P_1(\n H), \n H\ran dM} &=&\displaystyle{ -\int_M H^{1+2q}\vp^2L_1(H)dM}\\
&&\\
&&\displaystyle{- 2\int_M H^{1+2q}\vp\lan P_1(\n H), \n \vp\ran dM.}\\
\end{array}
$$
By using inequality (\ref{ineq.sobolev.1}), we have
\[
\begin{split}
(1+2q)\int_M H^{2q}\vp^2\lan P_1(\n H), \n H\ran dM&\leq  - \int_M H^{1+2q}\vp^2 L_1(H)dM\\
&+ \beta\int_M H^{2q}\vp^2\lan P_1(\n H),\n H\ran\\
&+ \dfrac{1}{\beta}\int_MH^{2+2q}\lan P_1(\n \vp),\n \vp\ran dM,\\
\end{split}
\]
i.e.,
\[
\begin{split}
\left(1+2q-\beta\right)\int_M H^{2q}\vp^2\lan P_1(\n H),\n H\ran dM&\leq -\int_MH^{1+2q}\vp^2 L_1(H)dM\\
&+\dfrac{1}{\beta}\int_MH^{2+2q}\lan P_1(\n \vp),\n\vp\ran dM.
\end{split}
\]
On the other hand, is well known, see \cite{AdCC}, Lemma $3.7$, that
$$
- L_1(H) = |\n H|^2 - |\n A|^2 - 3HK.
$$
Since $P_1$ is positive definite, we have
$$
\lan P_1(\n H),\n H\ran \leq (\tr P_1)|\n H|^2 = 2H|\n H|^2,
$$
i.e.,
$$
|\n H|^2\geq \dfrac{1}{2H}\lan P_1(\n H),\n H\ran.
$$
By using Proposition \ref{lema2} and inequality above, we obtain
$$
-L_1(H)\leq -\dfrac{2}{1+2c_0^2}|\n H|^2 - 3HK \leq -\dfrac{1}{(1+2c_0^2)H}\lan P_1(\n H),\n H\ran - 3HK.
$$
Then
\[\begin{split}
\left(1+\frac{1}{1+2c_0^2}+2q-\beta\right)\int_M H^{2q}\vp^2\lan P_1(\n H),\n H\ran dM &\leq 3\int_M H^{2+2q}(-K)\vp^2dM\\
&+ \dfrac{1}{\beta}\int_M H^{2+2q}\vp\lan P_1(\n \vp),\n \vp\ran dM.\\
\end{split}
\]
Replacing last inequality in (\ref{eq.b}), stability inequality becomes
\[
\begin{split}
3\int_M(-K)H^{2+2q}\vp^2 dM &\leq 3C_1\int_M H^{2+2q}\vp^2(-K) dM\\
&\qquad + C_2\int_{M}H^{2+2q}\lan P_1(\n\vp),\n\vp \ran dM,\\
\end{split}
\]
i.e.,
\[
3(1-C_1)\int_M H^{2+2q}(-K)\vp^2 dM \leq C_2 \int_M H^{2+2q}\lan P_1(\n \vp),\n \vp\ran dM.\\
\]
where 
$$
C_1=\dfrac{(1+q)^2+\beta(1+q)}{1+\frac{1}{1+2c_0^2}+2q-\beta},\ C_2=1 + \dfrac{(1+q)}{\beta} + \dfrac{(1+q)^2+(1+q)\beta }{\beta\left(1+\frac{1}{1+2c_0^2}+2q-\beta\right)},
$$ 
$0<q<\sqrt{\frac{1}{1+2c_0^2}}$ by hypothesis, and $\beta$ is taken such that $0<\beta<\dfrac{\frac{1}{1+2c_0^2}-q^2}{q+2}.$ This choice of $\beta$ is necessary to have $C_1<1.$ In fact,

\[
\begin{split}
\beta<\dfrac{\frac{1}{1+2c_0^2}-q^2}{q+2}&\Rightarrow q^2+\beta q + 2\beta <\frac{1}{1+2c_0^2}\\
&\Rightarrow (1+q)^2 + \beta(1+q) < 1+ \frac{1}{1+2c_0^2}+2q-\beta \\
&\Rightarrow C_1=\dfrac{(1+q)^2 + \beta(1+q)}{1+\frac{1}{1+2c_0^2}+2q-\beta}<1.\\
\end{split}
\]
Therefore,
$$
\int_M H^{2+2q}(-K)\vp^2dM\leq \dfrac{C_2}{3(1-C_1)}\int_M H^{2+2q}\lan P_1(\n \vp),\n \vp\ran dM.
$$
On the other hand, since $P_1$ is positive definite, we have 
$$
\lan P_1(\n\vp),\n\vp\ran\leq (\tr P_1)|\n H|^2\leq 2H|\n\vp|^2.
$$
Denoting by $C_3 =\dfrac{2C_2}{3(1-C_1)},$ we have 
$$
\begin{array}{rcl}
\displaystyle{\int_M H^{2+2q}(-K)\vp^2dM} &\leq& \displaystyle{\frac{C_3}{2}\int_M H^{2+2q}\lan P_1(\n \vp),\n \vp\ran dM}\\
&&\\
&\leq&\displaystyle{C_3\int_M H^{3+2q}|\n \vp|^2dM.}\\
\end{array}
$$
Letting $\vp=\psi^p,$ where $2p=5+2q,$ we obtain
\begin{equation}\label{ineq2}
\int_M H^{2+2q}(-K) \psi^{5+2q} dM \leq C_3p^2\int_M H^{3+2q} \psi^{3+2q}|\n \psi|^2 dM. 
\end{equation}
By using Young's inequality, i.e., 
$$
xy\leq \dfrac{x^a}{a} + \dfrac{y^b}{b},\ \dfrac{1}{a}+\dfrac{1}{b}=1
$$ 
with
$$
x=\delta H^{3+2q}\psi^{3+2q},\ \ y=\dfrac{|\n\psi|^2}{\delta}, \ \ a=\dfrac{5+2q}{3+2q}, \ \ b=\dfrac{5+2q}{2}, \ \ \mbox{and} \ \ \delta>0,
$$
we have
$$
H^{3+2q}\psi^{3+2q}|\n\psi|^2 \leq \dfrac{3+2q}{5+2q}\delta^{\frac{5+2q}{3+2q}}H^{5+2q}\psi^{5+2q} + \dfrac{2}{5+2q}\delta^{-\frac{5+2q}{2}}|\n\psi|^{5+2q}.
$$
Replacing last inequality in inequality (\ref{ineq2}), we obtain
$$
\begin{array}{rcl}
\displaystyle{\int_M H^{2+2q}(-K) \psi^{5+2q} dM} &\leq&\displaystyle{ \dfrac{3+2q}{5+2q}p^2C_3\delta^{\frac{5+2q}{3+2q}}\int_M H^{5+2q}\psi^{5+2q} dM}\\
&&\\
&&\displaystyle{ + \dfrac{2}{5+2q}p^2C_3\delta^{-\frac{5+2q}{2}}\int_M|\n\psi|^{5+2q}dM,}\\
\end{array}
$$
i.e.,
\begin{equation}\label{eq.e}
\int_M H^{5+2q}\left(\frac{(-K)}{H^3}-\Lambda_1\delta^{\frac{5+2q}{3+2q}}\right)\psi^{5+2q}dM\leq \Lambda_2\delta^{-\frac{5+2q}{2}}\int_M |\n\psi|^{5+2q}dM,
\end{equation}
where $\Lambda_1 = \dfrac{3+2q}{5+2q}p^2C_3$ and $\Lambda_2=\dfrac{2p^2}{5+2q}C_3.$
\end{proof}

\begin{remark}
In \cite{SSY}, Schoen, Simon, and Yau obtained the following Sobolev type inequality for minimal hypersurfaces $M^n$ immersed in $\R^{n+1}$:
\begin{equation}\label{ineq.sobolev.SSY}
\int_M |A|^{2p}\psi^{2p} dM \leq C(n,p)\int_M |\n \phi|^{2p} dM, 
\end{equation}
for $p\in[2,2+\sqrt{2/n}),$ and for all function $\psi:M\ria\R$ compactly supported on $M.$ By using inequality of Proposition \ref{prop.sobolev2}, we obtain a similar result for hypersurfaces $M^3$ immersed in $\R^4$ with zero scalar curvature. In fact, if $R=0,$ then $H^2=|A|^2.$ Choosing an orientation such that $H>0,$ we have $H=|A|.$ In this case, we have

\begin{corollary}[Sobolev type inequality]\label{sobolev-SSY}
Let $x:M^3\ria\R^4$ be a stable isometric immersion with zero scalar curvature and such that $\dfrac{(-K)}{H^3}\geq c>0$ everywhere. Then, for all smooth function $\psi$ with compact support in $M,$ for all $\delta>0$ and $p\in\left(\frac{5}{2},\frac{5}{2}+\sqrt{\frac{1}{1+2c_0^2}}\right),$ there exists a constant $C(p)>0$ such that
\begin{equation}\label{ineq.sobolev}
\int_M |A|^{2p}\psi^{2p} dM \leq C(p)\int_M |\n \psi|^{2p}dM.
\end{equation}
\end{corollary}
\end{remark}

\begin{remark}
In the recent article \cite{ilias-nelli-soret}, Ilias, Nelli, and Soret, obtained results in this direction for hypersurfaces with constant mean curvature.
\end{remark}

Now let us prove the Theorem A stated in the Introduction.

\begin{teoA}
There is no stable complete hypersurface $M^3$ of $\R^4$ with zero scalar curvature, polynomial volume growth and such that $$\dfrac{(-K)}{H^3}\geq c>0$$ everywhere, for some constant $c>0.$
\end{teoA}

\begin{proof}
Suppose by contradiction there exists a complete stable hypersurface attending conditions of Theorem A. Then we can apply Proposition \ref{prop.sobolev2}. Choose the compact supported function $\psi:M\ria\R$ defined by
\begin{equation}
\label{test-function}
\psi(\rho(p))=\left\{
\begin{array}{ccl}
1&\mbox{if}& p\in B_r;\\
&&\\
\dfrac{2r-\rho(p)}{r}&\mbox{if}& p\in B_{2r}\backslash B_r;\\
&&\\
0&\mbox{if}&p\in M\backslash B_{2r},\\
\end{array}\right.
\end{equation}
where $\rho(p)=\rho(p,p_0)$ is the distance function of $M.$ By using this function $\psi$ in the inequality of Proposition \ref{prop.sobolev2}, we have
\begin{equation}\label{eq.A}
\begin{split}
\displaystyle{\int_{B_r} H^{5+2q}\left(\frac{(-K)}{H^3}-\Lambda_1\delta^{\frac{5+2q}{3+2q}}\right)dM}&\leq \displaystyle{\int_{B_{2r}} H^{5+2q}\left(\frac{(-K)}{H^3}-\Lambda_1\delta^{\frac{5+2q}{3+2q}}\right)\psi^{5+2q}dM}\\
&\leq \displaystyle{\Lambda_2\delta^{-\frac{5+2q}{2}}\int_{B_{2r}} |\n\psi|^{5+2q}dM}\\
&\leq \displaystyle{\Lambda_2\delta^{-\frac{5+2q}{2}}\dfrac{\vol B_{2r}}{r^{5+2q}},}\\
\end{split}
\end{equation}
for $0<q<\sqrt{\frac{1}{1+2c_0^2}}.$
Taking $\delta>0$ sufficiently small and since, by hypothesis, $\dfrac{(-K)}{H^3}\geq c>0,$ we get
$$
\left(\dfrac{(-K)}{H^3}-\Lambda_1\delta^{\frac{5+2q}{3+2q}}\right)>0.
$$
By hypothesis, $M$ has polynomial volume growth. It implies that
$$
\lim_{r\ria\infty}\dfrac{\vol(B_r)}{r^\al}<\infty, \ \alpha\in(0,4].
$$
Letting $r\ria\infty$ in the inequality $(\ref{eq.A}),$ we obtain
\[
\begin{split}
\lim_{r\ria\infty}\int_{B_r} H^{5+2q}\left(\frac{(-K)}{H^3}-\Lambda_1\delta^{\frac{5+2q}{3+2q}}\right)dM &\leq \Lambda_2\lim_{r\ria\infty}\dfrac{\vol(B_{2r})}{r^\al}\cdot\lim_{r\ria\infty}\dfrac{1}{r^{5+2q-\al}}=0.
\end{split}
\]
Therefore $H\equiv0,$ and this contradiction finishes the proof of the theorem.
\end{proof}

\begin{remark}\label{rem1}
{\normalfont
In the proof of Theorem A, $M$ need not even be properly immersed, since we are taking intrinsic (geodesic) balls. Since $M$ is complete, we have $M=\bigcup_{n=1}^\infty B_{r_n}$ for some sequence $r_n\ria\infty,$ and thus we can take $r\ria\infty$ in the estimate.
}
\end{remark}

\begin{remark} 
By using their Sobolev inequality $(\ref{ineq.sobolev.SSY})$, Schoen, Simon, and Yau gave a new proof of Bernstein's Theorem for dimension less than or equal to $5,$ namely, that the only entire minimal graphs $M^n$ in $\R^{n+1}, \ n\leq 5$ are hyperplanes. By using our version of Sobolev inequality $(\ref{ineq.sobolev})$, we prove the following Bernstein type result.
\end{remark}

As a corollary of Theorem A, we have the following result.
 
\begin{teoB}\label{teo333}
There are no entire graphs $M^3$ of $\R^4$ with zero scalar curvature and such that $$\dfrac{(-K)}{H^3}\geq c>0$$ everywhere, for some constant $c>0.$
\end{teoB}

\begin{proof}
Suppose there exists an entire graph $M$ satisfying the conditions of corollary. In \cite{ASZ}, Proposition $4.1,$ p. $3308,$ Alencar, Santos, and Zhou showed that entire graphs with zero scalar curvature and whose mean curvature does not change sign are stable. Since $R=0$ by hypothesis, we have $H^2=|A|^2.$ Provided $K\neq 0$ everywhere, we have $H^2=|A|^2>0$ which implies that $H$ does not change sign. Thus, the entire graph $M$ is stable. On the other hand, is well known that graphs satisfies $\vol(B_r)\leq Cr^4, C>0.$ Therefore, by using the hypothesis $\dfrac{-K}{H^3}\geq c>0,$ inequality (\ref{eq.A}), in the proof of Theorem A, p.\pageref{eq.A}, and taking $r \ria \infty $ we obtain the same contradiction.
\end{proof}


\section{Examples}
\label{sec:3}
The class of hypersurfaces treated here is non-empty, as shown in the following example. It can be found in \cite{hounie-leite}, Lemma 2.1, p. $400.$ See also \cite{AdCE}, p. $213-214$ and \cite{dCE}, p. $161.$
\begin{example}
{\normalfont
Let $M^3\hookrightarrow\R^4$ the rotational hypersurface parametrized by
$$
X(t,\theta,\vp)=(f(t)\sen\theta\cos\vp,f(t)\sen\theta\sen\vp,f(t)\cos\theta,t),
$$
where $f(t)=\dfrac{t^2}{4m}+m$ and $m$ is a non-negative constant. The principal curvatures are
$$
\lambda_1=\lambda_2=\dfrac{m^{1/2}}{f^{3/2}}, \ \lambda_3=-\dfrac{1}{2}\dfrac{m^{1/2}}{f^{3/2}}.
$$ 
Then $R=0$ and $\dfrac{-K}{H^3}=\dfrac{4}{27}$ everywhere. Since $M^3$ is a rotational hypersurface and its profile curve is quadratic, it has polynomial volume growth. Then by Theorem A the immersion is unstable. 
}
\end{example}

This example appears in the Theory of Relativity as the embedding of the space-like Schwarzschild manifold of mass $m/2>0,$ see Introduction of \cite{bray}, for details.


The following class of hypersurfaces are well known, see \cite{AdCE}, p. $214,$ and they are the classical examples of stable hypersurfaces with zero scalar curvature. This class show us that some condition over nullity of Gauss-Kronecker curvature are needed.

\begin{example}\label{ex2}
{\normalfont
Let $M^3\subset\R^4$ be the cylinder parametrized by
$$
{\bf x}(u,v,t)=(u,v,\al(t),\be(t)), u,v,t\in\R,
$$
where $c(t):=(\al(t),\be(t))$ is a parametrized curve with positive curvature $k(t)$
at every point. In this case, principal curvatures are
$$
\lambda_1=0, \lambda_2=0, \ \mbox{and} \ \lambda_3=k(t).
$$
Thus $R=0, \ H>0$ and $K=0$ everywhere. Then, $M^3$ is stable, see \cite{AdCE}.

Observe that if $c(t)=(t,f(t)),$ the cylinder $M$ is the graph of the smooth function $F:\R^3\ria\R$ given by $F(u,v,t)=f(t)$. In particular, taking $f(t)=t^2$ or $f(t)=\sqrt{1+t^2}$ we obtain an entire graph with polynomial volume growth, $R=0, H>0$ and $K=0$ everywhere.
}
\end{example}


\section{Non-embedded Tubes}
\label{sec:4}
Let $x:M^3\ria\R^4$ be an isometric immersion. Following Nelli and Soret, see \cite{nelli}, we define the \emph{tube of radius} $h$ \emph{around} $M$ the set
$$
T(M,h) = \{x\in\R^4;\ \exists\ p\in M, x=p+t\eta, t\leq h(p)\}, 
$$
where $\eta$ is the normal vector of the second fundamental form of $x,$ and $h:M\ria\R$ is an everywhere non-zero smooth function. If $|A|\neq0$ everywhere, we define the \emph{subfocal tube} the set
$$
T\left(M,\dfrac{\epsilon}{|A|}\right), \  0<\epsilon\leq1.
$$ 
Denote by $T(r,h)$ the tube of radius $h$ around $B_r\subset M,$ i.e., considering $M=B_r$ in the above definition, and let 
$$
V(r,h)=\int_{T(r,h)}dT,
$$ 
where $dT$ denotes the volume element of the tube. 
If $R=0,$ and choosing an orientation such that $H>0,$ we have $H=|A|.$ Under the conditions of Proposition \ref{prop.sobolev2}, and assuming that $\dfrac{(-K)}{H^3}\geq c>0,$ then there exists a constant $C(q)$ depending only on $0<q<\sqrt{\frac{1}{1+2c_0^2}}$ such that
\begin{equation}\label{ineq3}
\int_{B_r}|A|^{5+2q}\psi^{5+2q}dM\leq C(q)\int_{B_r}|\n\psi|^{5+2q}dM.
\end{equation}
 Choosing the same function with compact support used in the proof of Theorem A (see $(\ref{test-function}),$ p. $\pageref{test-function}$), we obtain
$$
\int_{B_r}|A|^{5+2q}dM \leq C(q)\dfrac{\vol(B_r)}{r^{5+2q}}.
$$
The following lemma is essentially the same Lemma $1$ of \cite{nelli}, p. $496,$ and the proof will be omitted here.
\begin{lemma}\label{lemma_nelli}
Let $M^3$ be a complete, stable hypersurface of $\R^4$ satisfying $R=0$ and $\dfrac{(-K)}{H^3}\geq c>0$ everywhere.
\begin{enumerate}
\item[(a)] For $r>0$ sufficiently large, there exists a constant $\alpha(q),$ depending only on $0<q<\sqrt{\frac{1}{1+2c_0^2}}$ such that
\begin{equation}\label{ineq5}
\vol(B_r)>\alpha(q)r^{5+2q}.
\end{equation}

\item[(b)] For each $\beta>1,$ $0<q<\sqrt{\frac{1}{1+2c_0^2}},$ and $r>0$ satisfying inequality $(\ref{ineq5})$ above, there exists a sufficiently large $\tilde{r}>r$ such that 
$$
\vol(B_{\tilde{r}}) - \vol(B_{\beta^{-1}\tilde{r}})>\alpha(q)r^{5+q}.
$$
\end{enumerate}
\end{lemma}

The next result is a vanishing scalar curvature version of Theorem $1,$ p. $499$ of \cite{nelli}. 

\begin{teoC}
Let $M^3$ be a stable complete hypersurface of $\R^4$ with vanishing scalar curvature. Suppose that the second fundamental form of the immersion is bounded and there exists a constant $c>0$ such that $\dfrac{(-K)}{H^3}\geq c>0$ everywhere. Then, for constants $0<b_1\leq1, \ b_2>0,$ and for any smooth function $h:M\ria\R$ satisfying
\begin{equation}\label{hyp.teoC}
h(p)\geq\inf\left\{\dfrac{b_1}{|A(p)|},b_2\rho(p)^\delta\right\},\ \delta>0,
\end{equation}
the tube $T(M,h)$ is not embedded. Here, $\rho(p)$ denotes the intrinsic distance in $M$ to a fixed point $p_0\in M.$
\end{teoC}
\begin{proof}
In \cite{nelli}, Nelli and Soret showed that

$$
V(r,h)=\int_{B_r} h(p)dM -\dfrac{1}{2}\int_{B_r}h(p)^2H(p)dM - \dfrac{1}{4}\int_{B_r}h(p)^4K(p)dM.
$$
By using the classical inequality between geometric and quadratic means, one finds that
$$
\begin{array}{rcl}
K&=&\lambda_1\lambda_2\lambda_3\leq |\lambda_1||\lambda_2||\lambda_3|\\
&&\\
&\leq&\left(\dfrac{\lambda_1^2+\lambda_2^2+\lambda_3^2}{3}\right)^{3/2}\\
&&\\
&=&\dfrac{1}{3\sqrt{3}}|A|^3,
\end{array}
$$
i.e.,
\begin{equation}\label{ineq7}
K(p)\leq \dfrac{1}{3\sqrt{3}}|A(p)|^3.
\end{equation}
Let $B_r^+$ the set where $\dfrac{b_1}{|A(p)|}$ is the infimum and $B_r^{-}=B_r\backslash B_r^+.$ Then
$$
\begin{array}{rcl}
V(r,h)&\geq&\displaystyle{ b_1\int_{B_r^+}\dfrac{1}{|A|}dM - \dfrac{b_1^2}{2}\int_{B_r^+}\dfrac{1}{|A|^2}HdM - \dfrac{b_1^4}{4}\int_{B_r^+}\dfrac{1}{|A|^4}KdM}\\
&&\\
&&\displaystyle{+b_2\int_{B_r^-}\rho^\delta dM - \dfrac{b_2^2}{2}\int_{B_r^-}\rho^{2\delta}HdM - \dfrac{b^4_2}{4}\int_{B_r^-}\rho^{4\delta} K dM.}\\
\end{array}
$$
Since $H=|A|,$ we have
$$
\begin{array}{rcl}
V(r,h)&\geq&\displaystyle{\left(b_1 - \dfrac{b_1^2}{2} - \dfrac{b_1^4}{12\sqrt{3}}\right)\int_{B_r^+}\dfrac{1}{|A|}dM}\\
&&\\
&&\displaystyle{+ b_2\int_{B_r^-}\rho^\delta dM - \dfrac{b_2^2}{2}\int_{B_r^-}\rho^{2\delta}HdM - \dfrac{b_2^4}{4}\int_{B_r^-}\rho^{4\delta}KdM}.
\end{array}
$$
Let us estimate the integrals over $B_r^-.$ By using inequality $(\ref{ineq7})$ above, we get
$$
- K\geq - \dfrac{1}{3\sqrt{3}}|A|^3\geq - \dfrac{b_1}{b_2}\dfrac{1}{3\sqrt{3}}\rho^{-3\delta}
$$
and
$$
-H=-|A|\geq\dfrac{b_1}{b_2}\rho^{-\delta}.
$$
By hypothesis, $|A|$ is bounded, then there exists $\displaystyle{a:=\inf_M \dfrac{1}{|A|}.}$ Therefore
$$
\begin{array}{rcl}
V(r,h)&\geq&\displaystyle{\left(b_1-\dfrac{b_1^2}{2} - \dfrac{b_1^4}{12\sqrt{3}}\right)\int_{B_r^+}\dfrac{1}{|A|}dM + \left(b_2-\dfrac{b_2b_1}{2} - \dfrac{b_2^3b_1}{12\sqrt{3}}\right)\int_{B_r^-}\rho^\delta dM}\\
&&\\
&\geq&\displaystyle{a\left(b_1-\dfrac{b_1^2}{2} - \dfrac{b_1^4}{12\sqrt{3}}\right)\vol(B_r^+) + \left(b_2-\dfrac{b_2b_1}{2} - \dfrac{b_2^3b_1}{12\sqrt{3}}\right)\int_{B_r^-}\rho^\delta dM}.\\
\end{array}
$$
On the other hand, for $r$ sufficiently large,
$$
\begin{array}{rcl}
\displaystyle{\int_{B_r^-}\rho^\delta dM} &=& \displaystyle{\int_{B_r^-\backslash B_{\beta^{-1}r}^-}\rho^\delta dM + \int_{B_{\beta^{-1}r}^-}\rho^\delta dM  \geq\int_{B_r^-\backslash B_{\beta^{-1}r}^-}\rho^\delta dM}\\
&&\\
&\geq&\displaystyle{\left(\dfrac{r}{\beta}\right)^\delta[\vol(B_r^-) - \vol(B_{\beta^{-1}r}^-)]}\\
&&\\
&\geq&\displaystyle{[\vol(B_r^-) - \vol(B_{\beta^{-1}r}^-)].}\\ 
\end{array}
$$
Then
$$
\begin{array}{rcl}
V(r,h)&\geq&a\left(b_1 - \dfrac{b_1^2}{2} -\dfrac{b_1^4}{12\sqrt{3}}\right)\left(\vol(B_r^+)-\vol(B_{\beta^{-1}r}^+)+\vol(B_{\beta^{-1}r}^+)\right)\\
&&\\
&&+\left(b_2-\dfrac{b_2^2b_1}{2} -\dfrac{b_2^3b_1}{12\sqrt{3}}\right)\left(\vol(B_r^-) -\vol(B_{\beta^{-1}r}^-)\right)\\
&&\\
&\geq&C[\vol(B_r) - \vol(B_{\beta^{-1}r})].\\
\end{array}
$$
By using Lemma \ref{lemma_nelli}, item $(b),$ there exists $\tilde{r}>r$ such that
\begin{equation}\label{ineq9}
V(\tilde{r},h)\geq C\tilde{r}^{5+q}.
\end{equation}
%
The Euclidean distance is less than or equal to the intrinsic distance. It implies
$$
B_r(p)\subset \B(p,r),
$$
where $B_r(p)\equiv B_r$ and $\B(p,r)$ denotes the intrinsic and the Euclidean ball of center $p$ and radius $r.$ By using $(\ref{hyp.teoC})$, we have
$$
h(q)\geq\min\left\{\dfrac{b_1}{|A|},b_2\rho(q)^\delta\right\}\geq\min\left\{\inf_M\dfrac{b_1}{|A|},b_2\rho(q)^\delta\right\}=\inf_M\dfrac{b_1}{|A|}=b_1a,
$$
for $0<b_1\leq1$ and $\rho$ sufficiently large, then
$$
T\left(r,b_1a\right)\subset T(r,h).
$$
Suppose, by contradiction, that $T(r,b_1a)$ is embedded. Since
$$
T(r,b_1a)\subset \B(p,r+2b_1a),
$$
then its volume $V(r,b_1a)$ satisfies
$$
V(r,b_1a)\leq \vol(\B(p,r+2b_1a))=\om_4(r+2b_1a)^4,
$$ 
where $\om_4$ is the volume of $\B(p,1).$ Let us consider two different cases. First, if $M$ is not contained in any ball, above inequality is a contradiction with $(\ref{ineq9})$ for $r$ sufficiently large. Therefore, $T(r,b_1a),$ and thus $T(r,h),$ is not embedded for $r$ sufficiently large. In the second case, if $M$ is contained in some ball, then $T(M,h)$ has finite volume (since $T(M,h)$ is embedded) and it is also a contradiction with $(\ref{ineq9}).$
\end{proof}

\section{Appendix}
\label{Appendix}

Let us prove the following fact established in the Introduction:

\emph{Let $x:M^3\ria\R^4$ be an isometric immersion with zero scalar curvature. If $H$ and $K$ denotes the mean curvature and Gauss-Kronecker curvature, respectively, then
$$
0\leq \dfrac{-K}{H^3}\leq\dfrac{4}{27}\ \mbox{everywhere on} \ M.
$$
}

\begin{figure}[ht]
\centering
\includegraphics[scale=0.35]{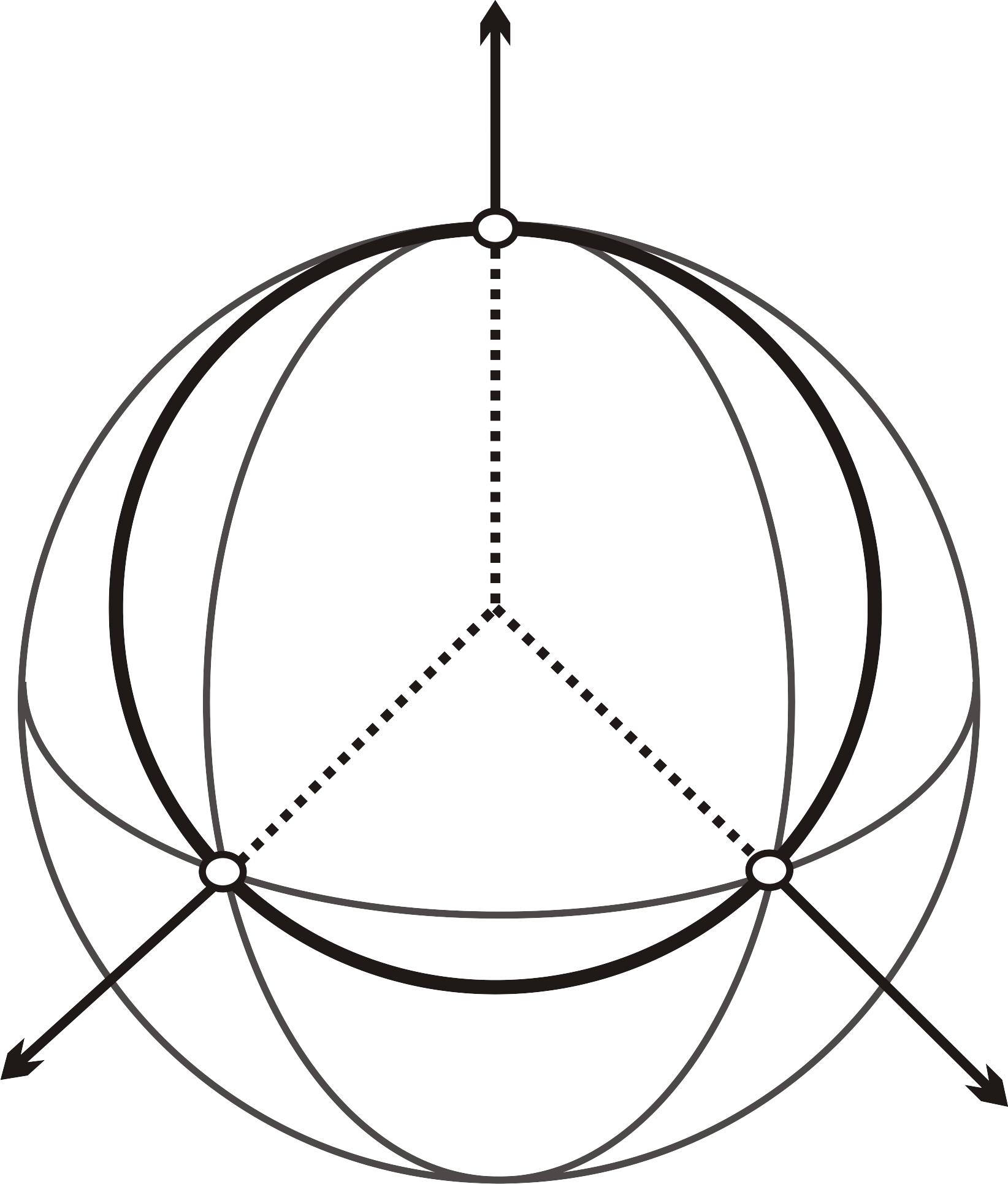}
\label{fig:1}
\caption{Representation of the domain $N_\om$ of $\dfrac{K}{H^3}$ over $\s^2,$ considering this function as an algebraic function of the eigenvalues. This domain is the intersection of one of the plane $\lambda_1+\lambda_2+\lambda_3=1$ with $\s^2.$ The hypothesis cuts off only three small neighbourhoods around the coordinate axis.}
\end{figure}

In fact, let $(\lambda_1,\lambda_2,\lambda_3)=t\omega$ where $\omega\in\s^{2}.$ By using $(\ref{eq.curv}),$ we can see that $R,H$ and $K$ are homogeneous polynomials. It implies $H(t\om)=tH(\om), \ R(t\om)=t^2R(\om),\ K(t\om)=t^3K(\om)$ and hence
$$
\dfrac{K}{H^3}(t\om)=\dfrac{K}{H^3}(\om).
$$
Then the behavior of $\frac{K}{H^3}$ depends only of its values on the sphere $\s^{2}.$ Since $N:=\{(\lambda_1,\lambda_2,\lambda_3)\in\R^3; R=\lambda_1\lambda_2+\lambda_1\lambda_3+\lambda_2\lambda_3=0\}$ is closed and $\s^{2}$ is compact, we obtain that $N_\om=N\cap\s^{2}$ is compact, see figure \ref{fig:1}. Then, $\frac{K}{H^3}:N_\om\ria\R$ is a continuous function with compact domain. The claim then follow from the Weierstrass maxima and minima theorem. Upper bound $\frac{4}{27}$ can be found by using Lagrange multipliers method.



\begin{minipage}[b]{0.98\linewidth}
\begin{flushright}
\def\arraystretch{1.2}
\begin{tabular}{l}
Greg\' orio Silva Neto\\
Universidade Federal de Alagoas,\\ 
Instituto de Matemática,\\ 
57072-900, Maceió, Alagoas, Brazil.\\
gregorio@im.ufal.br\\
\end{tabular}
\end{flushright}
\end{minipage} \hfill

\end{document}